\documentclass[reqno]{amsart}
\pdfoutput=1
\usepackage[utf8]{inputenc}
\usepackage{amsmath,calligra,mathrsfs}
\usepackage{amssymb}
\usepackage{amsthm}
\usepackage{tabularx}
\usepackage{hyperref}
\usepackage{tikz}
\usepackage[thinlines]{easytable}
\usetikzlibrary{cd}
\everymath{\displaystyle}
\usepackage[hmargin = 1in,vmargin=1in]{geometry}

\newcommand{\Sp}{\mathrm{Sp}}

\newcommand{\tr}{\mathrm{tr}}
\newcommand{\Mor}{\mathrm{Mor}}

\newcommand{\C}{\mathbb{C}}

\newcommand{\Het}{\mathrm{H}_{\mathrm{\acute{e}t}}}

\newcommand{\Hom}{\mathrm{Hom}}

\newcommand{\ZZ}{\mathbb{Z}}

\newcommand{\F}{\mathbb{F}}
\newcommand{\Oh}{\mathcal{O}}

\newcommand{\Pn}{\mathbb{P}}
\newcommand{\Pic}{\mathrm{Pic}}
\newcommand{\Q}{\mathbb{Q}}

\newcommand{\PGL}{\mathrm{PGL}}

\newcommand{\G}{\mathbb{G}}

\newcommand{\Sch}{\mathrm{Sch}}
\newcommand{\Set}{\mathrm{Set}}

\newcommand{\Ho}{\mathrm{H}}

\newcommand{\Gal}{\mathrm{Gal}}
\newcommand{\br}{\mathrm{Br}}

\usepackage{scalerel}
\usepackage{stackengine,wasysym}

\newcommand{\X}{\mathcal{X}}

\newcommand{\NS}{\mathrm{NS}}
\newcommand{\Z}
{\mathbb{Z}}

\newcommand{\Ss}{\mathcal{S}}
\usepackage[hmargin = 1in,vmargin=1in]{geometry}
\usepackage[backend=biber, sorting=anyvt]{biblatex}
\renewbibmacro{in:}{%
  \ifentrytype{article}{}{\printtext{\bibstring{in}\intitlepunct}}}

\DeclareMathOperator{\SHom}{\mathscr{H}\text{\kern -3pt {\calligra\large om}}\,}
\newtheorem{theorem}{Theorem}[section]

\newtheorem{lemma}[theorem]{Lemma}

\theoremstyle{definition}
\newtheorem{definition}[theorem]{Definition}

\theoremstyle{remark}
\newtheorem{remark}[theorem]{Remark}

\newtheorem{notation}[theorem]{Notation}

\usepackage{blindtext}
\usepackage{subfiles}
\usepackage{comment}
\usepackage{lipsum}
\title{K3 surfaces over $\mathbb{Q}$ of degree 10 that have Picard rank~1 }
\author{victor de vries}
\date{August 2025}

\begin{document}

\maketitle

\begin{abstract}
We give examples of K3 surfaces over $\Q$ of degree $10$ whose geometric Picard group has rank~$1$. These K3 surfaces are intersections in $\Pn^9$ of three hyperplanes, one quadric and the image of the Plücker embedding of the Grasmannian $\mathrm{Gr}(2,5)$. We also give an example of a K3 surface of degree $6$ over~$\Q$ whose Picard rank is $1$. 
\end{abstract}

Let $S$ be a K3 surface over a number field $k$. To be precise, $S$ is a surface, $S\to \Sp(k)$ is smooth, proper and geometrically connected, whose canonical sheaf $\omega_S$ is trivial and which satisfies $\Ho^1(S,\Oh_S)=0$. For arithmetic purposes, one of its most useful invariants is the rank of its geometric Picard group, $\mathrm{rk}\,\Pic(\overline{S})$, which we call the \textit{Picard rank of }$S$. For example, Bogomolov and Tschinkel \cite{Bogomolov1999DensityOR} proved that when the Picard rank is at least~$5$, the rational points on the K3 surface are potentially dense. The Picard rank $\rho$ satisfies $1\leq \rho\leq h^{1,1}(S_\C)=20$ and as a heuristic one can say that the lower $\rho$ is, the harder it is to understand the arithmetic of $S$. In particular if $S$ has Picard rank equal to $1$, then its arithmetic should be complicated and interesting.

Let $(S,D)$ be a pair, where $S$ is a K3 surface and $D$ is an ample divisor on $S$ whose class in $\Pic(\overline{S}$) is primitive. The pair is called \textit{of degree} $2d$ if the self-intersection number of $D$ is $2d$.  We may also refer to a K3 surface $S$ as being of degree $2d$ if it admits such a $D$ as above. A general K3 surface of degree $2d\leq 8$ is a complete intersection in $\Pn^{d+1}$ and so the result by Terasoma~\cite{Terasoma1985} implies that for $2d\leq 8$, there exists a K3 surface over $\Q$ of degree $2d$ with Picard rank $1$. Moreover, this can be strengthened to include the cases $d\in \{5,6,7,8,9\}$, because in these cases the general K3 surface is a complete intersection of sections of line bundles in a Grassmannian (which depends on $d$) and the argument of Terasoma goes through as shown in the proof of Proposition 3.14 in~\cite{MotivicInvariantsofBirMaps}. These proofs are however not constructive.

The first examples of K3 surfaces over $\Q$ with Picard rank $1$ were given by van Luijk~\cite{RonaldK3rk1}. These examples are quartics in $\Pn^3$, i.e. K3 surfaces of degree $4$. Later Elsenhans and Jahnel \cite{elsenhansDegree2Rk1} gave an example of a K3 surface over $\Q$ of degree $2$ with Picard rank $1$. By using Fourier-Mukai duality with a K3 surface of degree~$2$, the authors of \cite{BrGrpsonK3} were able to exhibit a K3 surface over~$\Q$ of degree $8$ whose Picard rank is $1$, see \cite[Section ~5.4]{BrGrpsonK3}.\\

The aim of this document is to exhibit K3 surfaces over $\Q$ of degree $10$ with Picard rank $1$. Moreover, as a sidenote we `fill in the gap' and exhibit K3 surfaces over $\Q$ of degree $6$ with Picard rank $1$. We note that a very recent publication \cite{AsherScheible} also gives examples of K3 surfaces over $\Q$ of degree $6$ whose Picard rank is $1$. In this document, the main focus is the degree $10$ case.

We consider the Grassmannian $\mathrm{Gr}(2,5)$, which embeds via the Plücker embedding in $\Pn^9$. A general K3 surface over $\C$ of degree~$10$ is the intersection in $\Pn^9$ of $\mathrm{Gr}(2,5)$, three hyperplanes and one quadric hypersurface. It is thus a complete intersection in $\mathrm{Gr}(2,5)$ of type $(1,1,1,2)$, see Definition \ref{Complint}.
We will give one K3 surface of this type over $\F_2$ and one over $\F_3$. The one over $\F_3$ will have Picard rank $2$ and the one over $\F_2$ is chosen such that a certain technical condition holds, see Lemma \ref{S2Prop}. The surface $S_2$ is the zero locus in $\mathrm{Gr}(2,5)$ of three linear forms $l_1,l_2,l_3$ and one quadratic form $q$ and similarly $S_3$ is cut out in $\mathrm{Gr}(2,5)$ by linear forms $l_1',l_2',l_3'$ and a quadratic form $q'$. These forms are given in Lemmas \ref{S2} and \ref{S3}. The Grassmannian $\mathrm{Gr}(2,5)$ has a canonical model over $\Z$, which is cut out by the `Plücker relations' in $\Pn^9_{\Z}$. Thus we may lift the forms over $\F_2,\F_3$ to forms over the integers to obtain a subscheme of $\mathrm{Gr}(2,5)_{\Q}$. Our main theorem is the following.

\begin{theorem}\label{maintheorem}
Let $S$ be a subscheme of $\mathrm{Gr}(2,5)_{\Q}$ that is the zero locus of forms $L_1,L_2,L_3,Q$, where the $L_i$ are lifts of $l_i,l_i'$ to forms over $\Z$ and $Q$ is a lift of $q,q'$ to a form over $\Z$. The scheme $S$ is a K3 surface over~$\Q$ of degree $10$, whose Picard rank is $1$.
\end{theorem}

The approach is inspired by the one taken by van Luijk \cite{RonaldK3rk1}, whose method also includes constructing surfaces modulo primes and then lifting. Moreover, we make use of a result by Elsenhans and Jahnel \cite{elsenhans2011picard}, which in our case implies that the lattice $\Pic(\overline{S})$ is a primitive sublattice of $\Pic(\overline{S_3})$. The K3 surfaces $S_2$ and $S_3$ are found by using Magma to pick random K3 surfaces until we find two for which Lemmas \ref{S2Prop} and \ref{S3Prop} hold. Therefore we suspect that examples of K3 surfaces over $\Q$ with Picard rank $1$ and degree $2d\in \{12,14,16,18\}$ can be found by using similar methods on the condition that the $\text{mod } p$ computations in Magma remain feasible. 

It is not clear how one would `lift' when the general K3 surface is not a complete intersection (which happens for $d\gg 0$). A conjecture of Shafarevich \cite{Shafarevich+1996+103+108} predicts that for every positive integer $e$, there are only finitely many lattices that appear as $\Pic(\overline{X})$ for $X$ a K3 surface over a number field $K$ with $[K:\Q]=e$. This conjecture would imply that there are only finitely many degrees $2d$ for which one can find a K3 surface over $\Q$ of degree $2d$ for which one has $\rho=1$.

\section{Preliminaries}

For any surface $X$, one can consider the Néron-Severi group $\mathrm{NS}(X)$ of $X$. After quotienting $\mathrm{NS}(X)$ by its torsion, it has the structure of a lattice under the intersection pairing. For a K3 surface $S$ we have $h^1(S,\Oh_S)=0$, which implies that we have $\Pic(S)=\mathrm{NS}(S)$ and that this is group is torsion-free.  We have the following lemma, which relates the Néron-Severi group of a variety to the Néron-Severi group of its reduction. 

\begin{lemma}\label{smproperbc}
Let $R$ be a DVR with $\Sp(R)=\{\eta,x\}$, where the point $x$ is closed. Let $\pi\colon \X\to \Sp(R)$ be a smooth and proper relative surface. There is a specialization map $\NS(\overline{\X_\eta})/\{\text{tors}\}\to \NS(\overline{\X_x})/\{\text{tors}\}$, which is an inclusion of lattices. Moreover, for $l$ a prime that is inverible in $R$, the inertia subgroup of $\Gal(\overline{\eta}/\eta)$ acts trivially on $\mathrm{NS}(\overline{\X_\eta})\otimes_\Z \Z_l$ and the specialization map is $\Gal(\overline{x}/x)$-equivariant.
\begin{proof}
For the statement about the inclusion of lattices, see \cite[Proposition 6.2]{vanluijk2004elliptick3surfaceassociated}. The Kummer sequence yields a commuting diagram in which each row is equivariant with respect to the relevant Galois action
\begin{equation}\label{NSdiag}
\begin{tikzcd}
\mathrm{NS}(\overline{\X_\eta})\otimes_\Z \Z_l \arrow[d] \arrow[r, hook] & {\Het^2(\overline{\X_\eta},\Z_l(1))} \arrow[d, "\sim"] \\
\mathrm{NS}(\overline{\X_x})\otimes_\Z \Z_l \arrow[r, hook]              & {\Het^2(\overline{\X_x},\Z_l(1))}          
\end{tikzcd}.
\end{equation}
For the second statement it suffices to show that the inertia group of $\Gal(\overline{\eta}/\eta)$ acts trivially on $\Het^2(\overline{X}_{\eta},\mu_{l^n})$ for all $n>0$.
The sheaf $R^2\pi_*\mu_{l^n}$ is locally constant by smooth-proper base change and thus we can choose a Galois étale algebra $R'/R$, such that $R^2\pi_*\mu_{l^n}|_{R'}$ is the sheaf associated to a constant abelian group $M$. Denote by $j'\colon \eta'\to \Sp(R')$ the inclusion of the generic point and by $\pi'\colon \X'\to \Sp(R')$ and $\pi_{\eta'}\colon \X'_{\eta'}\to \{\eta'\}$ the base-changes of~$\pi$. We have by smooth-base change the identification $R^2\pi_{\eta',*}\mu_{l_n}=j'^*R^2\pi'_*\mu_{l_n}=j'^*M$ and so $R^2\pi_{\eta',*}\mu_{l^n}$ is constant. 

The stalk of the sheaf $R^2\pi_*\mu_{l^n}$ at $\overline{\eta}$ corresponds to the $\Gal(\overline{\eta}/\eta)$-module $\Het^2(\overline{X}_{\eta},\mu_{l^n})$ by proper base-change. By the last sentence of the previous paragraph, the field extension $\kappa(\eta')/\kappa(\eta)$ trivializes this Galois module and since $\kappa(\eta')/\kappa(\eta)$ is unramified we conclude that the inertia group of $\Gal(\overline{\eta}/\eta)$ acts trivially on $\Het^2(\overline{X}_{\eta},\mu_{l^n})$ for every $n$ and so on $\mathrm{NS}(\overline{\X_{\eta}})\otimes_\Z \Z_l$.

The specialization map $\Het^2(\overline{\X}_{\eta},\mu_{l^n})\to\Het^2(\overline{\X}_{x},\mu_{l^n})$ may be understood as follows: After passing to the finite étale cover $R'/R$, the sheaf $R^2\pi_*\mu_{l^n}|_{R'}$ is constant and so the stalks of $R^2\pi_*\mu_{l^n}$ are identified in a canonical way. This identification is compatible with the action of $\pi_1(\Sp(R),\overline{\eta})$ on $R^2\pi_*\mu_{l^n}$ and $\Gal(\overline{x}/x)$ acts on these stalks in the same way, via $\pi_1(\Sp(R),\overline{\eta})$. Together with Diagram (\ref{NSdiag}) this implies the last statement of the lemma.
\end{proof}
\end{lemma}

Typically when $\mathrm{Gr}(2,5)$ is mentioned, it is implicit that it is a variety over a field. There is a natural model of $\mathrm{Gr}(2,5)$ over $\Z$, which has the following functor of points.

\begin{definition}
Let $\G$ be the functor \[(\Sch)^{\mathrm{op}}\to\Set\quad S\mapsto \{\text{locally free sheaves of rank }2,\text{ }\mathcal{E}\subset \Oh_S^{\oplus 5}\,|\,\Oh_S^{\oplus 5}/\mathcal{E}\text{ is locally free}\}.\]
\end{definition}

\begin{remark}
The functor $\G$ is represented by a scheme $\mathrm{Gr}(2,5)$ that is smooth and proper over $\Z$. To ease notation we write $\G=\mathrm{Gr}(2,5)$. 
There is a closed embedding $\mathrm{Plu}\colon \G\to \Pn^9$, that can be described on $L$-points, where $L$ is a field by, $[W\subset L^{\oplus 5}]\mapsto [\det(W)\subset \bigwedge^2 L^{\oplus 5}\cong L^{\oplus 10}]$. The image of $\G$ under this embedding is cut out by the Plücker relations $\{G_i=0\}_{1\leq i\leq 5}$, which are

\[G_1:=x_{12}x_{34}+x_{23}x_{14}-x_{13}x_{24}\]
\[G_2:=x_{12}x_{35}+x_{23}x_{15}-x_{13}x_{25}\]
\[G_3:=x_{12}x_{45}+x_{24}x_{15}-x_{14}x_{25}\]
\[G_4:=x_{13}x_{45}+x_{34}x_{15}-x_{14}x_{35}\]
\[G_5:=x_{23}x_{45}+x_{34}x_{25}-x_{24}x_{35}.\]
We may sometimes also write $\G$ for a base-change $\G_k$ of $\G$, where $k$ is a field. What is meant with $\G$ should be clear from the context.
\end{remark}

\begin{definition}\label{Complint}
Denote by $\Oh_{\G}(n)$ the invertible sheaf $\mathrm{Plu}^*\Oh(n)$ on $\G$. We say that a variety $X\subset \G_k$ of dimension $6-t$ is a \textit{complete intersection of type} $(n_1,...,n_t)$ if we have $X=V_\G(s_1,...,s_t)$ for some $s_i\in \Oh_\G(n_i)(\G)$.
\end{definition}

We need the following lemma in order to be able to prove Lemma \ref{comlpeteintisK3}, which is key for our construction.

\begin{lemma}
The canonical sheaf of $\G$ is isomorphic to $\Oh_\G(-5)$.
\end{lemma}

This lemma should be well-known (it follows e.g. from \cite[Appendix B.5.8]{fulton2012intersection}), however we could not find a simple proof for it in the literature that works in positive characteristic.

We make a remark about the existence of certain universal sheaves on $\G$.

\begin{remark}
The identity map $\G\to \G$ corresponds to a rank $2$ locally free sheaf $\mathcal{R}\subset \Oh_{\G}^{\oplus 5}$. This sheaf is universal in the sense that for any morphism $f\colon Y\to \G$ corresponding to $\mathcal{E}\subset \Oh_Y^{\oplus 5}$, we have $\mathcal{E}=f^*\mathcal{R}$. We denote by $\mathcal{Q}$ the universal quotient $\Oh_\G^{\oplus 5}/\mathcal{R}$, which is a locally free sheaf of rank $3$.
\end{remark}

We can use $\mathcal{R}$ and $\mathcal{Q}$ to describe the tangent sheaf of $\G$. This formula is known to experts (cf. \cite[Appendix B.5.8]{fulton2012intersection}), but we include it to make the exposition self-contained.

\begin{lemma}\label{tangent sheaf}
The tangent sheaf of $\mathbb{G}$ is isomorphic to $\SHom(\mathcal{R},\mathcal{Q})$.
\begin{proof}
Let $\mathbb{T}\to \mathbb{\G}$ be the tangent bundle of $\G$. It has as its universal property that for any $\G$-scheme $Y\to \G$, we have $\Mor_{\G}(Y,\mathbb{T})=\Mor_{\G}(Y[\epsilon],\G)$. A morphism $f\colon S\to \G$ corresponds to a rank $2$ locally free sheaf $\mathcal{E}\subset \Oh_Y^{\oplus 5}$ with locally free cokernel for which we have $\mathcal{E}=f^*\mathcal{R}$ and $\Oh_Y^{\oplus 5}/\mathcal{E}=f^*\mathcal{Q}$. 

Suppose that we have an $\Oh_Y$-module homomorphism $\phi\colon f^*\mathcal{R}\to f^*\mathcal{Q}$. Then we can associate to it the locally free sheaf $\mathcal{E}_\phi \subset \Oh_{Y[\epsilon]}^{\oplus 5}$, which is on local sections given by the span $\langle r+\epsilon \tilde{\phi(r)},\epsilon r\rangle$, where with the tilde we mean any lift to a section of $\Oh_{Y}^{\oplus 5}$ (the module is independent of these lifts). The module $\mathcal{E}_\phi$ corresponds to an element in $\Mor_\G(Y[\epsilon],\G)$ and the assignment $\phi\mapsto \mathcal{E}_\phi$ induces a morphism of functors $\Psi\colon \SHom(\mathcal{R},\mathcal{Q})\to \Mor_\G((-)[\epsilon],\G)= \Mor_\G(-,\mathbb{T})$. 

Both the domain and the codomain of the morphism of functors give locally free sheaves on $\G$ when restricting the domain of the functors from $(\Sch/\G)^{\mathrm{op}}$ to $(\mathrm{Open}(\G))^{\mathrm{op}}$. To show the map between these locally free sheaves is an isomorphism, it is sufficient to show that it is an isomorphism on the fiber over every closed point $x$ of $\G$. Let $[W]\colon \Sp(L)\to \G$ be such a point. The morphism of functors evaluated at this point $[W]$ is: \[\Psi_{[W]}\colon \Hom(W,L^5/W)\to \{\text{rank }2 \text{ free }M\subset L[\epsilon]^{\oplus 5}\,|\, \text{cokernel is free}\text{ and }M[\epsilon]\overset{\epsilon\to 0}{\to} W\}\quad \phi\mapsto W_\phi.\]
We first show that $\Phi_{[W]}$ is injective. Let $\phi_1,\phi_2\colon W\to L^5/W$ be linear maps that satisfy $\phi_1(w)\neq \phi_2(w)$ for some $w\in W$. For any lift $\tilde{\phi_1(w)}$ of $\phi_1(w)$ to $L^5$ we have that $r+\epsilon \tilde{\phi_1(w)}$ lies in $W_{\phi_1}$, but it does not lie in $W_{\phi_2}$. For the surjectivity, let $M$ be a module in the codomain of $\Psi_{[W]}$ and choose $m=m_1+\epsilon m_2$ and $n=n_1+\epsilon n_2$ with $m_i,n_i\in L^{\oplus 5}$ such that $\{m,n\}$ is an $L[\epsilon]$ basis of $M$. The $2$ by $5$ matrix $(m_2,n_2)$ defines a linear map $g\colon W\to L^{\oplus 5}$. Postcomposing $g$ with the quotient map to $L^{\oplus 5}/W$ gives a map $\overline{g}$ such that $M=W_{\overline{g}}$ holds and so $\Psi_{[W]}$ is surjective.

We conclude that for any closed point $[W]$ of $\G$, the map $\Psi_{[W]}$ is an isomorphism and so $\Psi$ induces an isomorphism of locally free sheaves $\SHom(\mathcal{R},\mathcal{Q})\to \Mor_\G(-,\mathbb{T})$.
\end{proof}
\end{lemma}

Now we have the following lemma on the Plücker embedding.

\begin{lemma}\label{pluckersheaf}
The invertible sheaves $\det(\mathcal{Q})$ and $\Oh_\G(1)$ are isomorphic.
\begin{proof}
The exact sequence of sheaves $0\to \mathcal{R}\to \Oh_{\G}^{\oplus 5}\to \mathcal{Q}\to 0$ induces a surjection $\bigwedge^3\Oh_{\G}^{\oplus 5}\overset{\pi}{\to} \det(\mathcal{Q})$ and thus a morphism $\iota\colon \G\to \Pn^9$. We will show that this morphism equals the Plücker embedding. Under $\iota$ a closed point $[W]\colon \Sp(L)\to \G$ is mapped to the following one-dimensional subspace of $(\bigwedge^3L^{\oplus 5})^{\vee}$: \begin{equation}\label{Pluckerintrinisc}[W]\mapsto \{f\in (\bigwedge^3L^{\oplus 5})^{\vee}\,|\,f\text{ vanishes on all }s\in \bigwedge^3L^{\oplus 5}\text{ for which }\pi(s)_{[W]}=0\text{ holds}\}\in \Pn^9(L)\end{equation}
The assignment $s\mapsto \pi(s)_{[W]}$ is simply given by the natural map $\bigwedge^3L^{\oplus 5}\to \det(L^{\oplus 5}/W)$, which follows from the definition of the sheaf $\mathcal{Q}$. The wedge pairing $\bigwedge^3L^{\oplus 5}\otimes\bigwedge^2L^{\oplus 5}\to \det(L^{\oplus 5})$ identifies both factors in the domain with each others dual. So we can rewrite Equation (\ref{Pluckerintrinisc}) to
\begin{equation}\label{Pluckerrewritten}
[W]\mapsto \{f\in \bigwedge^2L^{\oplus 5}\,|\,f\wedge s=0\text{ for all }s\in \bigwedge^3L^{\oplus 5}\text{ that vanish in }\det(L^{\oplus 5}/W)\}.\end{equation}
The mentioned wedge pairing induces a perfect pairing $\det(W)\otimes \det(L^{\oplus 5}/W)\to \det(L^{\oplus 5})$. So we see that $\det(W)\subset  \bigwedge^2L^{\oplus 5}$ is contained in the subspace to which $[W]$ is mapped in Equation (\ref{Pluckerrewritten}). Since this subspace is $1$-dimensional, it equals $\det(W)$ and so the surjection $\bigwedge^3\Oh_\G^{\oplus 5}\to \det(\mathcal{Q})$ induces the Plücker embedding. We conclude that the invertible sheaves $\Oh_\G(1)$ and $\det(\mathcal{Q})$ are isomorphic.
\end{proof}
\end{lemma}

Now we can give an expression for the canonical sheaf of $\G$.

\begin{lemma}\label{canonicalbundleGr}
The canonical sheaf of $\G$ is isomorphic to $\Oh_\G(-5)$.
\begin{proof}
By Lemma \ref{tangent sheaf} the tangent sheaf of $\G$ is isomorphic to $\SHom(\mathcal{R},\mathcal{Q})$ and so the sheaf of differentials $\Omega_{\G/k}$ is isomorphic to $\mathcal{R}\otimes \mathcal{Q}^\vee$. So the canonical sheaf $\omega_{\G/k}$ is isomorphic to $\det(\mathcal{R}\otimes \mathcal{Q}^\vee)$, which is isomorphic to $\det(\mathcal{R})^{\otimes 3}\otimes (\det(\mathcal{Q})^{\vee})^{\otimes 2}$. The exact sequence $0\to \mathcal{R}\to \Oh_{\G}^{\oplus 5}\to \mathcal{Q}\to 0$ identifies $\det(\mathcal{R})$ with $\det(\mathcal{Q})^\vee$ (again under the wedge pairing) and so $\omega_{\G/k}$ is isomorphic to $(\det(\mathcal{Q})^\vee)^{\otimes 5}$. By Lemma \ref{pluckersheaf} this is isomorphic to $\Oh_{\G}(-5)$.
\end{proof}
\end{lemma}

We use this lemma to prove the following.

\begin{lemma}\label{comlpeteintisK3}
Let $X$ be a smooth variety over a field $k$ such that $X$ is a complete intersection of type $(1,1,1,2)$ in $\G$. Then $X$ is a K3 surface.
\begin{proof}
The Enriques-Severi-Zariski Lemma \cite[\href{https://stacks.math.columbia.edu/tag/0FD8}{Tag 0FD8}]{stacks-project} implies that an ample divisor on a proper scheme of dimension $\geq 2$ is connected. Applying this lemma four times gives that $S$ is geometrically connected. To show that $S$ is simply connected we apply Lefschetz's theorem for the fundamental group \cite[\href{https://stacks.math.columbia.edu/tag/0ELE}{Tag 0ELE}]{stacks-project}. This theorem implies that if $X$ is regular, connected and has dimension $\geq 3$, then for any ample divisor $D$ on $X$ we have $\pi_1(D)\cong \pi_1(X)$. Applying this theorem four times gives $\pi_1(S)\cong \pi_1(\G)$. Since $\G$ is a rational variety, it is simply connected and so we deduce that $\pi_1(S)$ is trivial.

It remains to be proven that $S$ has trivial canonical bundle.
The adjunction formula for $\iota\colon S\hookrightarrow \G$ gives us $\omega_S=\iota^*\omega_{\G}\otimes \det(\mathcal{N}_{S/\G})$. We have $\omega_{\G}=\Oh_{\G}(-5)$ by Lemma \ref{canonicalbundleGr}. We have $\mathcal{N}_{S/\G}=(\mathcal{I}/\mathcal{I}^2)^*$ for $\mathcal{I}\subset \Oh_{\G}$ the ideal sheaf of $S$. Since $S$ is the zero locus of a section $s$ of the vector bundle $\mathcal{E}:=\Oh_{\G}(1)^{\oplus 3}\oplus \Oh_{\G}(2)$ of rank $4$ ($=\dim(\G)-\dim(S)$), the ideal $\mathcal{I}$ is resolved by its Koszul complex:
\begin{equation}
    ...\to \bigwedge^2\mathcal{E}^*\overset{i(s)}{\to} 
\mathcal{E}^*\overset{s^*}{\to} \mathcal{I}\to 0
\end{equation}
Here $s^*$ is the dual of $s\colon \Oh_S\to \mathcal{E}$ and $i(s)$ is given on sections by $v_1\wedge v_2\mapsto s^*(v_1)\cdot v_2-s^*(v_2)\cdot v_1$. Applying~$\iota^*$ to the exact sequence transforms $i(s)$ into the zero-map and transforms $\mathcal{I}$ into $\mathcal{I}/\mathcal{I}^2$. We thus obtain an isomorphism of $\Oh_S$-modules $\mathcal{I}/\mathcal{I}^2\cong \iota^*(\Oh_{\G}(-1)^{\oplus 3}\oplus \Oh_{\G}(-2))$ and therefore we obtain $\det(\mathcal{N}_{S/\G})\cong \iota^*\Oh_{\G}(5)$. We conclude by using the adjunction formula that we have $\omega_S\cong \Oh_S$.
\end{proof}
\end{lemma}

\section{The surfaces $S_2$ and $S_3$}
In this subsection we exhibit K3 surfaces $S_2/\F_2$ and $S_3/\F_3$ with certain properties. These properties will allow us to show that suitably chosen simultaneous lifts to $\Q$ of $S_2$ and $S_3$ have Picard rank $1$.\\

We start by exhibiting our $S_2$.

\begin{lemma}\label{S2}
Let $S_2$ be the scheme in $\G_{\F_2}$ defined by $V(l_1,l_2,l_3,q)$ where the $l_i$ and $q$ are:
\begin{align*}
&l_1 = x_{12}+x_{15} + x_{24} + x_{34}\\
&l_2 = x_{13}+x_{24} + x_{25}+ x_{34}\\
&l_3 = x_{14}+x_{23} + x_{24} + x_{35}\\
&q = x_{12}^2+ x_{14}^2 + x_{13}x_{15} + x_{13}x_{23} + x_{15}x_{23} + x_{12}x_{24}+ x_{13}x_{24} +x_{15}x_{24} +x_{24}^2 + x_{12}x_{25}
 +x_{15}x_{25} \\&
 \indent + x_{23}x_{25} + x_{24}x_{25} + x_{25}^2 + x_{13}x_{34} + x_{14}x_{34} + x_{15}x_{34} + x_{23}x_{34} + x_{24}x_{34} + x_{25}x_{34} +x_{34}^2
+ x_{13}x_{35}\\
 &  \indent + x_{14}x_{35} + x_{15}x_{35} + x_{23}x_{35} + x_{25}x_{35} + x_{34}x_{35} + x_{35}^2 + x_{12}x_{45} + x_{14}x_{45} + x_{34}x_{45} + x_{45}^2.\end{align*}
The scheme $S_2$ is a K3 surface.
\begin{proof}
We verify with Magma that $S_2$ is smooth and of dimension $2$. Now the lemma follows from Lemma~\ref{comlpeteintisK3}.
\end{proof}
\end{lemma}

For $X=V(F_1,...,F_r)\subset \Pn^m$ a projective variety we denote by $X_{\text{sing}}$ the singular subscheme of $X$, i.e. the subscheme of $X$ that is defined as the vanishing locus of all of the maximal minors of the $r$ by $(m+1)$ matrix $(\frac{\partial F_i}{\partial x_j})$. The key property that $S_2$ satisfies is the following.

\begin{lemma}\label{S2Prop} Every hyperplane section $H$ of $S_2$ is geometrically reduced. For every such $H$ for which $H_{\text{sing}}$ the base change of $H$ to $\F_4$ has an irreducible component of degree $\geq 6$.
\begin{proof}
Direct verification with Magma. One runs through all options for $H$.
\end{proof}
\end{lemma}

Now we exhibit our $S_3$.

\begin{lemma}\label{S3}
Let $S_3$ be the subscheme $V_\G(l'_1,l'_2,l'_3,q')$ of $\G_{\F_3}$, where $l'_1,l'_2,l'_3,q'$ are:
\begin{align*}
&l'_1=x_{12} + 2x_{24} + 2x_{25} + 2x_{35}\\
&l'_2=x_{13} + x_{15} + 2x_{24} + 2x_{25} + x_{34} + 2x_{35}\\
&l'_3=x_{14} + x_{23} + x_{24} + 2x_{25} + 2x_{34}\\
&q'=x_{15}(x_{23} +x_{24} + x_{15} + x_{34} + x_{35}+ 2x_{45})+ x_{24}(2x_{25} + 2x_{34}   + x_{45}).\end{align*}
The scheme $S_3$ is a K3 surface.
\begin{proof}
We verify with Magma that $S_3$ is smooth and of dimension $2$. The lemma now follows from Lemma~\ref{comlpeteintisK3}.
\end{proof}
\end{lemma}

\begin{remark}
The equations for $S_3$ are chosen in such a way that the map $\pi\colon S_3\to \Pn^1$, that is defined on coordinates by $[x_{15}:x_{24}]$ is an elliptic fibration.  \end{remark}

We exploit the map $\pi\colon S_3\to \Pn^1$ to efficiently count $\F_{3^n}$ points on $S_3$ for $n\leq 10$.

\begin{lemma}\label{Points}
The number $\F_{3^n}$ points on $S_3$ for $n\leq 10$ is:\\
\begin{center}
\begin{TAB}(r,0.5cm,1cm)[5pt]{c|c|c|c|c|c|c|c|c|c|c|}{c|c}
 $n$ & $1$ & $2$ & $3$ & $4$ & $5$ & $6$ & $7$ & $8$ & $9$ & $10$\\
 $\#S_3(\F_{3^n})$  & $16$ & $94$ & $730$ & $6850$ & $58591$ & $533332$ & $4777705$ & $43057090$ & $387492661$ & $3486840049$  \\
\end{TAB}
\end{center}
\begin{proof}
This can be computed with Magma.
\end{proof}
\end{lemma}

We have the following lemma on the traces of an endomorphism of a vector space.

\begin{lemma}\label{Newton}
Let $V$ be a vector space of dimension $d$ and let $\varphi\colon V\to V$ be an endomorphism. Let $t_n$ denote the trace of $\varphi^{ n}$. Let $P(t)=t^d+a_1t^{d-1}+...+a_{d-1}t+a_d$ be the characteristic polynomial of~$\varphi$. The coefficients of $P(t)$ satisfy the recursive formula $a_1=-t_1$ and $-nc_n=\sum_{i=1}^{n-1}a_it_{k-i}$ for $n>1$.
\begin{proof}
This is called Newton's identity, which is well known.
\end{proof}
\end{lemma}

The point count of Lemma \ref{Points} can be used to compute the geometric Picard number of $S_3$. The key property of $S_3$ is that this number is $2$.

\begin{lemma}\label{S3Prop}
The K3 surface $S_3$ has Picard rank $2$.
\begin{proof}
A
By the Lefschetz fixed point formula we have the equality \begin{equation}\label{traceformula}\#S_3(\F_{3^n})=1+3^{2n}+\tr((\varphi^*)^{ n}\,|\,\Het^2(\overline{S_3},\Q_l)),\end{equation} where $l$ is a prime that is different from $3$ and $\varphi\colon\overline{S_3}\to \overline{S_3}$ is the relative Frobenius endomorphism. So by Lemma \ref{Points} we can compute $\tr((\varphi^*)^n\,|\,\Het^2(\overline{S_3},\Q_l))$ for $n\leq 10$. Let $\psi$ be the endomorphism of $\Het^2(\overline{S_3},\Q_l(1))$ that is induced by the relative Frobenius. The eigenvalues of $\psi^{ n}$ are exactly $\tfrac{\lambda_1}{3^n},...,\tfrac{\lambda_{22}}{3^n}$, where $\lambda_1,...,\lambda_{22}$ are the eigenvalues of $(\varphi^*)^n$. So we can compute $\tr(\psi^{ n}\,|\,\Het^2(\overline{S_3},\Q_l(1)))$ for $n\leq 10$.

Let $P(t)$ be the characteristic polynomial of $\psi$. The vector space $\NS(\overline{S_3})\otimes_\Z \Q_l$ includes $\psi$-equivariantly into $\Het^2(\overline{S_3},\Q_l(1))$. The cycle class of a hyperplane section and the cycle class of a fiber of $\pi\colon S\to \Pn^1$ are linearly independent in $\NS(\overline{S_3})$. Moreover these classes are fixed by $\psi$ and so $P(t)$ will have a factor $(t-1)^2$. We can thus write $P(t)=(t-1)^2\cdot Q(t)$ and we denote $Q(t)=t^{20}+c_1t^{19}+...+c_{19}t+c_{20}$ and $P(t)=t^{22}+a_1t^{21}+...+a_{21}t+a_{22}$. Knowing $\tr(\psi^{ n})$ for $n\leq 10$ allows us to compute $a_n$ for $n\leq 10$ by Lemma \ref{Newton}. This allows us to compute the coefficients $c_n$ for $n\leq 10$, which are:
\begin{center}
\begin{TAB}(r,0.5cm,1cm)[5pt]{c|c|c|c|c|c|c|c|c|c|c|}{c|c}
 $n$ & $1$ & $2$ & $3$ & $4$ & $5$ & $6$ & $7$ & $8$ & $9$ & $10$\\
 $c_n$  & $0$ & $\tfrac{1}{3}$ & $\tfrac{2}{3}$ & $-\tfrac{1}{3}$ & $1$ & $0$ & $\tfrac{2}{3}$ & $\tfrac{2}{3}$ & $-\tfrac{1}{3}$ & $1$  \\
\end{TAB}
\end{center}

By Poincaré duality, the polynomial $P(t)$ is either symmetric or anti-symmetric. If it were anti-symmetric, then $Q(t)$ would also be anti-symmetric, which is not the case because we have $c_{10}\neq 0$. So $P(t)$ and $Q(t)$ are symmetric, which implies that we have
\[P(t)=(t-1)^2(t^{20} + \tfrac{1}{3}t^{18} + \tfrac{2}{3}t^{17} - \tfrac{1}{3}t^{16} + t^{15} + \tfrac{2}{3}t^{13} + \tfrac{2}{3}t^{12} - \tfrac{1}{3}t^{11} + t^{10} - \tfrac{1}{3}t^9 + \tfrac{2}{3}t^8 + \tfrac{2}{3}t^7 + t^5
    - \tfrac{1}{3}t^4 + \tfrac{2}{3}t^3 + \tfrac{1}{3}t^2 + 1).\]
A Magma calculation shows that the second factor has no roots of unity as its roots. So $\psi$ has only two eigenvalues that are roots of unity. The $\psi$-orbit of any element in $\mathrm{NS}(\overline{S_3})$ is periodic and so the rank of $\mathrm{NS}(\overline{S_3})$ is bounded by the number of eigenvalues of $\phi$ that are roots of unity, which is $2$. Because the class of a fiber of $\pi:S_3\to \Pn^1$ and the class of a hyperplane section are linearly independent, the rank is also at least $2$, so we conclude that the rank of $\mathrm{NS}(\overline{S_3})$ is $2$.
\end{proof}
\end{lemma}

The rank $2$ lattice $\Lambda$ that is spanned by the class $H$ of a hyperplane section of $S_3$ and the class $M$ of a fiber of $\pi:S_3\to \Pn^1$ embeds into $\Pic(\overline{S_3})$. This lattice $\Lambda$ is represented by the Gramm-matrix $\begin{pmatrix}10 & 5 \\ 5 & 0\end{pmatrix}$ with respect to the basis $H,M$.

\begin{lemma}\label{latticeS3}
The lattice $\Lambda$ is equal to $\Pic(\overline{S_3})$.
\begin{proof}
Since the discriminant of $\Lambda$ is $-25$, the index of $\Lambda$ in $\Pic(\overline{S_3})$ is either $1$ or $5$. If it is $5$, then there is a class $N\in \Pic(\overline{S_3})$ such that $5N=mH+nM$ with $0\leq m,n\leq 4$ and $m,n$ both not zero. Taking self-intersection on both sides yields $5N^2=2m(m+n)$ and so we have $m=0$ or $5|m+n$. 

In the case $m=0$, we obtain $5N=nM$ with $1\leq n\leq 4$ and so $M$ is divisible by $5$ and thus we may assume that we have $5N=M$. By Riemann-Roch for surfaces, we see that either $N$ or $-N$ is effective and since $N\cdot H=1$ holds we see that $N$ is effective. But then $H\cdot N=1$ implies that $N$ is represented by a line $L$ that satisfies $L^2=0$, which is a contradiction because smooth rational curves on K3 surfaces have self-intersection $-2$. In the case that $5|m+n$, we get $5N=mH+(5-m)M$ with $1\leq m\leq 4$. So we see that $H-M$ is divisible by $5$ and therefore we may assume that we have $5N=H-M$. By doing precisely the same steps as in the previous case, we see that $N$ is represented by a line $L$ that satisfies $L^2=0$, which is again a contradiction. We conclude that $\Lambda$ is equal to $\Pic(\overline{S_3})$.
\end{proof}
\end{lemma}

\section{The degree $10$ example} In this subsection we prove that a certain type of simultaneous lift of $S_2$ and $S_3$ to $\Q$ has Picard rank $1$.

\begin{notation}
Recall the notation of $l_i$ for $1\leq i\leq 3$ and $q$ from Lemma \ref{S2} and recall the notation for $l_i'$ for $1\leq i\leq 3$ and $q'$ from Lemma \ref{S3}. For $1\leq i\leq 3$ we let $L_i$ be a linear form in $\Z[x_{12},...,x_{45}]$ that is a simultaneous lift of $l_i$ and $l_i'$. We let $Q$ be a quadratic form in $\Z[x_{12},...,x_{45}]$ that is a simultaneous lift of $q$ and $q'$. We denote by $\Ss$ the subscheme $V(L_1,L_2,L_3,Q)$ of $\G_{\Z}$ and we denote by $S$ the generic fiber of $\Ss$ over $\Sp(\Z)$.
\end{notation}

We first show that $S$ is a K3 surface.

\begin{lemma}\label{Dimension}
The scheme $S$ is a K3 surface.
\begin{proof}
The structure morphism $\pi\colon\Ss\to \Sp(\Z)$ has $S$ as its generic fiber. Since $\pi$ is proper, the map $\Sp(\ZZ)\to \ZZ_{\geq 0}\quad x\mapsto \dim \pi^{-1}(x)$ is upper semicontinuous. Thus $S_2$, the fiber over $(2)$, having $\dim(S_2)=2$ implies that $S$ has dimension at most $2$ and since $S$ is cut out by $4$ equations in $\G_{\Q}$, which has $\dim(\G_{\Q})=6$, we have $\dim(S)=2$. Since smoothness is an open property, we obtain that $S$ is smooth, because $S_2$ is smooth. We invoke Lemma \ref{comlpeteintisK3} to finish the proof.
\end{proof}
\end{lemma}

Now we prove the main theorem.

\begin{theorem}
The K3 surface $S$ has Picard rank $1$.
\begin{proof}

The scheme $\Ss$ is a model of $S$ over $\Sp(\Z)$, that is smooth at $2,3$, because the fibers at these points are $S_2,S_3$. By Lemma \ref{smproperbc}, we get an inclusion of lattices $\Pic(\overline{S})\hookrightarrow \Pic(\overline{S_3})$. Thus $\rho(\overline{S})$ is bounded by $\rho(\overline{S_3})$, which is $2$ by Lemma \ref{S3Prop}. Assume for a contradiction that $\rho(\overline{S})$ equals $2$.\\

By the result of Elsenhans and Jahnel \cite[Theorem 1.4]{elsenhans2011picard}, the cokernel of this inclusion is torsion-free, so we make the identification $\Pic(\overline{S})=\Pic(\overline{S_3})$. By Lemma \ref{latticeS3} we have that $\Pic(\overline{S_3})$ is spanned by the class $H_3$ of a hyperplane section of $S_3$ and the class $M_3$ of a fiber of $\pi:S_3\to \Pn^1$. This implies that $\Pic(\overline{S})$ is represented by the Gramm-matrix $\begin{pmatrix}10 & 5 \\ 5 & 0\end{pmatrix}$ with respect to the basis $H,M$, which are the classes that correspond to $H_3,M_3$.

For any field $k$, write $\Gamma_k$ for its absolute Galois group. By Lemma \ref{smproperbc} the inclusion $\Pic(\overline{S})\hookrightarrow \Pic(\overline{S_2})$ is $\Gamma_{\F_2}$-equivariant. Since there are only two classes $N\in \Pic(\overline{S})$ that satisfy $N^2=0$ and $N\cdot H=5$, the $\Gamma_{\F_2}$-orbit of $M$ has size at most $2$. In what follows we also write $H,M$ for the images of $H,M$ in $\Pic(\overline{S_2})$.

We first look at the case in which the orbit has size $1$. The Hochschild-Serre spectral sequence gives an exact sequence
\[0\to \Pic(S_2)\to \Pic(\overline{S_2})\to \br(\F_2)\to ....\]
Since the Brauer group of a finite field is trivial, we have $\Pic(\overline{S_2})^{\Gamma_{\F_2}}=\Pic(S_2)$. So we can find a divisor $D$ of $S_2$ for which we have $[D]=M$. Since $M^2=0$ holds, Riemann-Roch for $S_2$ tells us that we either have $h^0(M)>0$ or $h^0(-M)=h^2(M)>0$. Clearly $-M$ is not effective (intersect with $H$). So we may pick an effective divisor $D_1$ on $S_2$ whose class is $M$. Since $(H-M)^2=0$ holds, we have by the same argument that $H-M$ or $M-H$ is effective and intersecting with $H$ implies that $H-M$ is effective. Pick an effective divisor $D_2$ on $S_2$ such that $[D_2]=H-M$. Then $D_1+D_2$ is effective and its class is $H$, i.e. $D_1+D_2$ is an $\F_2$-hyperplane section of $S_2$.

Suppose that $M$ has $\Gamma_{\F_2}$-orbit equal to $\{M,H-M\}$. Then $M$ is fixed by $\Gamma_{\F_4}$ and so by the same spectral sequence argument as in the previous paragraph, we can find an effective divisor $D$ of $S_4:=\Ss_{\F_4}$ for which we have $[D]=M$. Let $\sigma$ be the generator of $\Gamma_{\F_4/\F_2}$. The effective divisor $\sigma(D)$ has class $H-M$ and so $D+\sigma(D)$ is a hyperplane section of $S_2$. So in both cases we see that $M$ is represented by an effective $D$ over $\F_4$ that lies in a hyperplane section, which is defined over $\F_2$.\\

We call this hyperplane section $E$ for which we have $E_{\F_4}=D+N$. By Lemma \ref{S3Prop}, the hyperplane section~$E$ is geometrically reduced and so $D$ and $N$ have no common component. Moreover reducedness implies that $E_{\text{sing}}$ has dimension~$0$. There is a closed immersion of the scheme-theoretic intersection $D\cap N$ into $E_{\text{sing}}$. The scheme-theoretic intersection $D\cap N$ has degree $D\cdot N=M\cdot (H-M)=5$ and thus we have $5\leq \deg(E_{\text{sing}})$. On the other hand, the equality $E_{\F_4}=D+M$ says that the irreducible components of $E_{\F_4}$ have degree at most $\deg(D)=\deg(N)=5$. By Lemma \ref{S2Prop} the combination of the last two sentences is not possible.\\

Since we have reached a contradiction, we conclude that we do not have $\rho(\overline{S})=2$. By looking modulo $3$ we had already seen that we have $\rho(\overline{S})\leq 2$ and we conclude that we the Picard rank of $\overline{S}$ is $1$.
\end{proof}
\end{theorem}

\section{The degree $6$ example} As promised in the introduction, we also exhibit an example of a degree $6$ K3 surface over $\Q$ with Picard rank $1$. A general K3 surface of degree $6$ is a complete intersection of a quadric and a cubic in $\Pn^4$ on which we take coordinates $x,y,z,v,w$. The approach is similar as for the degree $10$ case and we begin with a K3 surface modulo $2$.

\begin{remark}
A smooth complete intersection of a quadric and a cubic is a K3 surface. Indeed, the argument in the first paragraph of the proof of Lemma \ref{comlpeteintisK3} shows that such a complete intersection is simply connected and the adjunction formula shows that the canonical bundle is trivial.
\end{remark}

\begin{lemma}\label{Degree6X2}
Let $X_2$ be the subscheme of $\Pn^4_{\F_2}$ that is cut out by the equations $f=vq_1+wq_2=0$, where we have
\begin{align*}f&=x^2 + xz + yz + z^2 + xv + yv + zv + v^2 + yw + zw + w^2\\
q_1&=x^2+y^2+yz+z^2+xv\\
q_2&=y^2+yz+z^2+xv+v^2+xw+zw+vw+w^2.\end{align*}
The scheme $X_2$ is a K3 surface and its Picard lattice $\Pic(\overline{X_2})$ is isomorphic to $\begin{pmatrix}6 & 4 \\ 4 & 0\end{pmatrix}$.
\begin{proof}
The class $H$ of a hyperplane section and the class $F$ of the curve $C:f=v=q_2=0$ generate a sublattice $\Lambda$ of $\Pic(\overline{X_2})$ that has intersection matrix $\begin{pmatrix}6 & 4 \\ 4 & 0\end{pmatrix}$. We check with Magma that the point count of $X_2$ over the fields $\F_{2^i}$ with $i\leq 10$ is given as follows:

\begin{center}
\begin{TAB}(r,0.5cm,1cm)[5pt]{c|c|c|c|c|c|c|c|c|c|c|}{c|c}
 $n$ & $1$ & $2$ & $3$ & $4$ & $5$ & $6$ & $7$ & $8$ & $9$ & $10$\\
 $\#S_3(\F_{2^n})$  & $7$ & $29$ & $97$ & $273$ & $1057$ & $3905$ & $16065$ & $64513$ & $264193$ & $1049089$  \\
\end{TAB}
\end{center}

By the same argument as in Lemma \ref{S3Prop}, we verify that the Picard rank of $X_2$ is $2$. It remains to be shown that $\Lambda$ is a primitive sublattice of $\Pic(\overline{X_2})$.

Suppose for a contradiction that $\Lambda$ is not a primitive sublattice of $\Pic(\overline{X_2})$. Then there is a class $E$ in $\Pic(\overline{X_2})$ for which we have $2E=mH+nF$ with $m,n\in \{0,1\}$ and both nonzero. Taking the self-intersection on both sides yields that $m$ is divisible by $2$. Thus we obtain that $F$ is divisible by $2$ and we may assume that $F=2E$. If $-E$ is effective, then so is $-F$ and so $F$ is trivial, which is not the case and so $-E$ can not be effective. By Riemann-Roch for surfaces, we thus obtain $h^0(E)\geq 2$ and so there is an effective divisor $D$ whose class is $E$. 

Let $\mathcal{I}_C$ be the ideal sheaf of the curve $C$ in the first line of this proof. Since $C$ is integral, the long exact sequence in cohomology \[0\to k\to \Ho^0(C,\Oh_C)\to \Ho^1(\overline{X}_2,\mathcal{I}_C)\to \Ho^1(\overline{X_2},\Oh)=0\]
shows that we have $h^1(\overline{X_2},\mathcal{I}_C)=0$ and so by Serre-duality we have $h^1(\overline{X_2},F)=0$. Riemann-Roch for surfaces shows that we have $h^0(\overline{X_2},F)=2$. Because we have $2E=F$ and both classes are effective, we get $2=h^0(F)\geq h^0(E)\geq 2$ and so $h^0(E)=2$ holds, so we have $H^0(\overline{X_2},\Oh(D))=H^0(\overline{X_2},\Oh(2D))$. Since $F$ has self-intersection $0$, the divisors $D$ and $2D$ are basepoint free, see \cite[Proposition 3.3.10]{Huybrechts_2016}. But the $\PGL_2$-classes of maps that they induce to $\Pn^1$ are the same, so we obtain that $D$ and $2D$ are linearly equivalent and so $E$ is trivial. This is a contradiction, because $2E=F$ is nontrivial.

We conclude that $\Lambda$ is a primitive sublattice of $\Pic(\overline{X_2})$, which concludes the proof.
\end{proof}
\end{lemma}

Now we exhibit a degree $6$ K3 surface over $\F_3$ with a different Picard lattice, but the same Picard rank.

\begin{lemma}\label{Degree6X3}
Let $X_3$ be the scheme in $\Pn^4_{\F_3}$ that is given by $V(vl_1+wl_2,g)$, where $l_1,l_2,g$ are as follows:
\begin{align*}l_1&=x+2y+z\\
l_2&= 2y + z\\
g&=x^3 + 2xy^2 + y^3 + 2xyz + 2yz^2 + 2z^3 + 2x^2v + xyv + xzv + 2yzv + z^2v + 2xv^2 + 2yv^2 + v^3 \\&\indent + 2x^2w + y^2w + xzw + z^2w + 2xvw + 2yvw + 2v^2w + 2yw^2 + vw^2 + 2w^3\end{align*}
The scheme $X_3$ is a K3 surface of degree $6$, whose Picard lattice $\Pic(\overline{X_3})$ is isomorphic to $\begin{pmatrix}
6 & 3\\ 3 & 0
\end{pmatrix}$.
\begin{proof}
The class of a hyperplane section and the class of the curve $C:l_1=l_2=g=0$ generate a sublattice $\Lambda$ of $\Pic(\overline{X_3})$, that is isomorphic to $\begin{pmatrix}
6 & 3\\ 3 & 0
\end{pmatrix}$. By using Magma, we obtain that the point count of $X_3$ over the fields $\F_{3^i}$ with $1\leq i\leq 10$ is given as follows:

\begin{center}
\begin{TAB}(r,0.5cm,1cm)[5pt]{c|c|c|c|c|c|c|c|c|c|c|}{c|c}
 $n$ & $1$ & $2$ & $3$ & $4$ & $5$ & $6$ & $7$ & $8$ & $9$ & $10$\\
 $\#S_3(\F_{3^n})$  & $15$ & $95$ & $765$ & $6767$ & $59190$ & $531911$ & $4799838$ & $43054727$ & $387489285$ & $3486956120$  \\
\end{TAB}
\end{center}

By the same argument as in Lemma \ref{S3Prop}, we verify that the Picard rank of $X_3$ is $2$ and so $\Lambda$ has finite index in $\Pic(\overline{X_3})$. It remains to be shown that $\Lambda$ is a primitive sublattice of $\Pic(\overline{X_3})$.

Suppose for a contradiction that this is not the case. By a $\text{mod }3$ computation, we obtain that there is a class $E\in \Pic(\overline{X_3})$ such that we have $3E=H+F$ or such that we have $3E=F$. If we had $3E=H+F$, then $\Pic(\overline{X_3})$ is isomorphic to $\begin{pmatrix}0 & 1\\ 1& 0\end{pmatrix}$, but this is impossible by \cite[(5.4)]{VANGEEMENBrGrpsofK3}.  Since the curve $C$ with class $F$ is irreducible, then we get in similar fashion as in the last paragraph of the proof of Lemma \ref{Degree6X2} that $F$ and $E$ define the same morphism $\overline{X_3}\to \Pn^1$. Then we have $E=F$, which implies that we have $2F$ is trivial, a contradiction, because $\Pic(\overline{X_3})$ is torsion-free.

We conclude that $\Lambda$ is a primitive sublattice, which finishes the proof.
\end{proof}
\end{lemma}

With these $\text{mod } p$ examples in hand, we can lift them to obtain the theorem.

\begin{notation}
Denote the cubic in Lemma \ref{Degree6X2} by $h$ and denote the quadric in Lemma \ref{Degree6X3} by $q$.
\end{notation}

\begin{theorem}
Let $F,G\in \Z[x,y,z,v,w]$ be homogeneous polynomials such that $F$ reduces to $f,q$ and such that $G$ reduces to $h,g$. Let $X$ denote the subscheme $V(F,G)$ of $\Pn^4_{\Q}$, which is a lift of $X_2$ and $X_3$. The scheme $X$ is a K3 surface of degree $6$ with Picard rank $1$.
\begin{proof}
The same argument from Lemma \ref{Dimension} implies that $X$ is a K3 surface. By Lemma \ref{smproperbc}, the Picard rank of $X$ is bounded by $\mathrm{rk}\, \Pic(\overline{X_2})=2$. Suppose for a contradiction that the Picard rank of $X$ is $2$. By Lemma~\ref{smproperbc} we have that $\mathrm{disc}(\Pic(\overline{X_2}))=-16$ divides $\mathrm{disc}(\Pic(\overline{X}))$. On the other hand, Lemma \ref{smproperbc} combined with the result of Elsenhans and Jahnel \cite[Theorem 1.4]{elsenhans2011picard} implies that we have an isometry $\Pic(\overline{X_3})\cong \Pic(\overline{X})$ and so we have $\mathrm{disc}(\Pic(\overline{X}))=-9$. This is a contradiction and thus we conclude that we have $\mathrm{rk}\, \Pic(\overline{X})=1$.
\end{proof}
\end{theorem}

\subsection*{Acknowledgements}
The author thanks prof. van Luijk for sharing his helpful ideas, for interesting discussions and for his encouragement. The author thanks Wim Nijgh for giving coding advice for Magma. The author thanks prof. Shinder and prof. Várilly-Alvarado for sharing useful remarks.

\printbibliography

\end{document}